\documentclass[pamqarxiv,keywordsasfootnote]{ipartarxiv}

\RequirePackage{amsthm,amssymb}
\RequirePackage[leqno]{amsmath}
\RequirePackage{hyperref}

\startlocaldefs
\newtheorem{thm}{Theorem}

\endlocaldefs

\newtheorem{remark}{Remark}

\newcommand{\veps}{\varepsilon}

\def\tr{\textmd{tr}}

\def\Ric{\textmd{Ric}}

\def\L{\mathcal{L}}

\def\R{\mathbb{R}}

\def\R{\mathbb{R}}

\def\S{\Sigma}
\def\lf({\left(}
\def\ri){\right)}

\def\={\stackrel{(n=2)}{=}}
\def\p{\partial}

\def\H{\mathcal{H}}

\newcommand{\be}{\begin{equation}}
\newcommand{\ee}{\end{equation}}
\newcommand{\bee}{\begin{equation*}}
\newcommand{\eee}{\end{equation*}}

\newcommand{\m}{\mathfrak{m}}

\newcounter{mnotecount}[section]

\allowdisplaybreaks

\begin{document}

\begin{frontmatter}
	\title[]{On the evolution of the spacetime Bartnik mass}
	\begin{aug}
		\author{\fnms{Stephen} \snm{McCormick}\ead[label=e1]{stephen.mccormick@math.uu.se}}
		\address{Matematiska institutionen \\ Uppsala universitet, 751 06 Uppsala \\
			Sweden\\
			\printead{e1}}
		\and
		\author{\fnms{Pengzi} \snm{Miao}\thanksref{t3}\ead[label=e2]{pengzim@math.miami.edu}}
		\address{Department of Mathematics \\
			University of Miami, Coral Gables, FL 33146\\
			USA\\
			\printead{e2}}            
		\thankstext{t3}{Research partially supported by Simons Foundation 
		Collaboration Grant for Mathematicians \#585168.}
	\end{aug}

%\received{\sday{3} \smonth{1} \syear{2018}}

\begin{abstract}
It is conjectured that the full (spacetime) Bartnik mass of a surface $\Sigma$ is realised as the ADM mass of some stationary asymptotically flat manifold with boundary data prescribed by $\Sigma$. Assuming this holds true for a 
$1$-parameter family of surfaces $\Sigma_t$ evolving in an initial data set {with the dominant energy condition}, we compute an expression for the derivative of the Bartnik mass along these surfaces. 
An immediate consequence of  this formula is that the Bartnik mass of $ \Sigma_t$ is monotone non-decreasing 
whenever $ \Sigma_t$ flows outward. 

It is our pleasure to dedicate this paper to Robert Bartnik on the occasion of his $60$th birthday.
\end{abstract}


%%  Upper case for every keyword

\end{frontmatter}
 
\section{Introduction}
The problem of quasi-local mass in general relativity is the problem of determining a reasonable notion of mass 
associated to a closed spacelike $2$-surface. 
Over the years, there have been many candidates for a suitable quasi-local mass, and while each of them has physical motivations, they do not all agree in general. One important
definition is that due to Robert Bartnik \cite{Bartnik-89}, which inspired by the notion of electrostatic capacity, is defined as 
the infimum of the ADM mass of suitable asymptotically flat extensions of the surface.

A Riemannian $3$-manifold $(M,g)$ is said to be asymptotically flat (with one end) if, after excising a compact set, it is diffeomorphic to $\R^3$ minus a closed ball with appropriate decay on the metric. The standard decay conditions are that the metric approaches the flat metric near infinity at a rate of $|x|^{-1/2-\veps}$, with its first two derivatives each decaying one power of $|x|$ faster, and that the scalar curvature be integrable. Asymptotically flat initial data for the Einstein equations is then a triple $(M,g,K)$ where $K_{ij}$ is a symmetric tensor decaying at a rate of $|x|^{-3/2-\veps}$ with its first derivative decaying one power of $|x|$ faster.

Under such conditions, it is well-known \cite{Bartnik-86,Chrusciel-86} that the ADM energy and momentum are well-defined, 
and indeed the energy is independent of coordinates while the linear momentum transforms appropriately as a vector in $\R^3$, under changes of coordinates.
 In rectangular coordinates near infinity, the ADM energy can be computed using the standard expression \cite{ADM}:
\bee 
	E_{ADM}=\frac{1}{16\pi}\lim_{R\to\infty}\int_{S_R} \p_i g_{ij}-\p_j g_{ii}\, dS^j,
\eee
while the linear momentum is computed as
\bee 
	p_i=\frac{1}{8\pi}
	\lim_{R \to\infty}\int_{S_R} 
	\lf(K_{ij}-g_{ij}\tr_g K\ri) dS^j,
\eee 
where $S_R:=\{|x|=R \}$ denotes a large coordinate sphere. The spacetime positive mass theorem 
\cite{Schoen-Yau81, Witten81} then says $E^2\geq p_ip^i$, and 
the total ADM mass is defined by $\m_{ADM}=\sqrt{E^2-p_i p^i}$.
The Bartnik mass of a closed two-surface $\S$ bounding a domain $\Omega$ in an initial data set, is then taken to be
\bee 
	\m_B(\S)=\inf\{ \m_{ADM}(M,g,K):(M,g,K) \text{ is an }\textit{admissible extension } \text{of } \S \}.
\eee 
Here an {\em admissible extension} refers to an asymptotically flat initial data set that extends $\Omega$ in an appropriate way, satisfying the positive mass theorem. 
Bartnik conjectured that the above infimum is realised by initial data corresponding to a stationary vacuum solution; that is, vacuum Killing initial data (KID). 
In order to better explain what constitutes an admissible extension and KID, 
we must first introduce the Einstein constraint equations. We therefore reserve discussion of this until the following section.

It should be remarked that a significant portion of the literature to date focusses only on the time-symmetric case. In which case, one expects that the infimum is realised by a static metric; this is the crux of the static metric extension conjecture. However, here we would like to consider the full spacetime definition of Bartnik mass.

A more recent definition of quasi-local mass that possesses significant promise
is that due to M.-T. Wang and S.-T. Yau \cite{WangYau-PRL, WangYau-09}, which is based on a Hamiltonian analysis. 
An intriguing question is whether there exists a relationship between the Bartnik mass 
and the Wang-Yau mass.
In \cite{CWWY}, the Wang-Yau mass with reference to static spaces 
was introduced by P.-N Chen, M.-T. Wang, Y.-K. Wang, and S.-T. Yau. 
In the time-symmetric setting, recent work by S. Lu and the second-named author \cite{LM} 
 indicates that, if the static metric extension conjecture holds, the derivative of the Bartnik mass along an evolving family of surfaces agrees with the derivative of the Wang-Yau mass with reference to the static metric extension of the given surface.
In making this observation,  the derivative formula of the Bartnik mass in time-symmetric initial data (see \cite{Miao-ICCM})
plays a key role.

In this article, we present a computational formula for the derivative of the full spacetime Bartnik mass, under the assumption that the Bartnik mass is achieved and is differentiable.
	The main result is the following.

 \begin{thm}\label{thm-intro-main}
 	Let $(M,g,K)$ be an initial data set for the Einstein equations. 
	Let $\{ \S_t \}$ be a family of closed, embedded surfaces evolving in  $M$. 
	We assume the evolution is given in terms of a smooth map  $X:\S \times I \to M$ by
 	\be
 	\frac{d X}{dt}=\eta n .
 	\ee 
	Here $I $ is an interval, $n$ is the unit normal pointing towards infinity in $M$, and 
 	$\eta $ denotes the  speed of $\Sigma_t = X ( \Sigma, t)$. 
 	
 	Suppose that for each $\S_t$ there exists an admissible extension (in the sense of Section \ref{SSetup}) $(M_t,g_t,K_t)$ realizing the Bartnik mass of $\S_t$ that is stationary and vacuum. Moreover, suppose $\{ (M_t, g_t, K_t) \}_{t \in I} $  depends smoothly on $t$. Denote by $N_t,X_t^A,X_t^\nu$ the projections of the stationary Killing field orthogonal to the initial data slice, tangential to $\S_t$, and orthogonal to $\S_t$ in $M_t$, respectively.
 	
 	Then the evolution of the Bartnik mass is given by
 	\begin{equation} \label{eq-intro-evoformfull}
 	\begin{split}		\frac{d}{dt}\m_B(\S_t)=&\, \frac{1}{16\pi}\int_{\S_t}\eta N_t\lf( |\Pi_t^{(M)}-\Pi_t^{(S)}|^2+|K_{t\,\S}^{(M)}-K_{t\,\S}^{(S)}|^2 \ri)\, d\mu_t \\
 	&+\frac{1}{8 \pi}\int_{\S_t} \eta X_t^{\nu}\lf(K_{t\,\S}^{(M)}-K_{t\,\S}^{(S)} \ri)\cdot \lf(\Pi_t^{(M)}-\Pi_t^{(S)} \ri)\,  d\mu_t \\
 	&+\int_{\S_t}\eta\lf( N_t\rho+X_t^A J_A+X_t^\nu J_n\ri)\,d\mu_t,
	\end{split}
 	\end{equation}
where the superscripts $(S)$ and $(M)$ refer to quantities defined on the stationary extension $M_t$ and on the original manifold $M$ respectively; $\Pi_t$ is the second fundamental form of $\S_t$ in $M_t$; $d\mu_t$ is the volume form of $\S_t$; a subscript ${\S}$ refers to the restriction to $\S$; and $(\rho,J_A,J_n)$ is the energy-momentum covector corresponding to $(M,g,K)$, projected tangentially ($J_A$) and orthogonally ($J_n$) to $\S$.
 \end{thm}
	Under the key assumptions imposed in Theorem \ref{thm-intro-main}, one sees that an immediate consequence of formula \eqref{eq-intro-evoformfull} is that, if $(M, g, K$) satisfies the dominant energy condition, then
	$ \displaystyle \frac{d}{dt}\m_B(\S_t) \ge 0 $
	along any $\{ \Sigma_t \} $ that flows outward.
	We now give a few remarks concerning the main assumptions. 
	
	\begin{remark}
		Existence and uniqueness of a stationary vacuum extension realizing the Bartnik mass is a fundamental question that has remained open 
		since the definition was proposed in \cite{Bartnik-89}.  In the time-symmetric setting, recent progress on static vacuum extensions 
		has been made by M. Anderson and  M. Khuri \cite{A-K}, and by M. Anderson and J. Jauregui \cite{A-J}. 
		In particular, examples of boundary surfaces with zero Bartnik mass that do not admit a mass minimizer have been constructed  by M. Anderson and J. Jauregui \cite{A-J}. 
	\end{remark}
	
	\begin{remark}
		The continuity and differentiability of the Bartnik mass is another challenging question that remains to be rigorously analyzed.  
		In the time-symmetric setting, partial results on the continuity of the Barntnik mass has been given by the first-named author \cite{McCormick-18}.
		In the general case, existence and smooth dependence of stationary vacuum extensions of boundary data  close to a round sphere in the 
		Minkowski spacetime $\R^{3,1}$ has been recently obtained by Z. An \cite{An-18}. 
		
		In the derivation of   \eqref{eq-intro-evoformfull}, as the proof in Section \ref{SEvoFormula} shows, one only needs the $1$-parameter family  $ \{ (g_t, K_t) \}_{t \in I} $  to be differentiable in an appropriate space of initial data (on a fixed manifold with boundary $\Sigma$) so that one can differentiate the Hamiltonian along this curve, for example in an appropriate weighted Sobolev or H\"older space
	\end{remark}

\begin{remark}
Formula \eqref{eq-intro-evoformfull}  was first found by Robert Bartnik and the second-named author in 2007 (see \cite[Section 3]{Miao-ICCM}). Unfortunately the drafts referred to in \cite{Miao-ICCM} were never completed. It is our pleasure to present this formula here, and dedicate it  to Robert on the occasion of his 60th birthday.
\end{remark}

This article is organized as follows: In Section \ref{SSetup}, we recall the Einstein constraint equations and some other basic definitions. In Section \ref{Svariation}, we compute the first variation of the Reggie-Teitelboim Hamiltonian and give this expression in a form promoting Bartnik's geometric boundary data. Then in Section \ref{SEvoFormula} we use the computations in Section \ref{Svariation} to derive the evolution formula \eqref{eq-intro-evoformfull}.

\section{Setup}\label{SSetup}
Let $(M,g)$ be a Riemannian $3$-manifold and $K_{ij}$ be a symmetric $2$-tensor on $M$. The constraint map $\Phi$ is given by
\begin{align} \begin{split}
\Phi_0(g,K):&=R(g)+(\tr_gK)^2-|K|^2\\
\Phi_i(g,K):&=2( \nabla^jK_{ij}-\nabla_i(\tr_gK)).
\end{split}
\end{align}
This allows one to write the Einstein constraint equations simply as
\begin{align*}
	\Phi_0(g,K)&=16\pi \rho\\
	\Phi_i(g,K)&=16\pi J_i,
\end{align*}
where $\rho$ and $J_i$ correspond to appropriate projections of a source energy-momentum tensor from the spacetime 4-manifold.

Naturally, the energy-momentum source terms should not be completely arbitrary; one usually imposes the dominant energy condition, which amounts to the condition $\rho^2\geq J_iJ^i$. 
This is a standard assumption under which the positive mass theorem holds. 

We now turn back to discuss the notion of admissible extensions, defining the Bartnik mass. Let $\S$ be a $2$-surface bounding some domain $\Omega$ in a given initial data set $(\widehat M,\widehat g,\widehat K)$. We would like to consider an admissible extension of $\S$ to be some initial data set $(M,g,K)$ with interior boundary 
$\p M$ isometric to $\S$, matching $\Omega$ in some sense. In particular, we would like to ask that the resultant manifold obtained by gluing $\Omega$ to $M$ along $\S$, satisfies the dominant energy condition. Of course, $g$ and $K$ are not necessarily smooth along $\S$, so the best we can ask for is that the dominant energy condition is satisfied distributionally.

Assuming momentarily that the data is smooth, 
in a neighborhood of $\S$, we can foliate $M$ by level sets of the distance function to $\S$ and 
let $H$ denote the mean curvature of each level set. It follows from the second variation of area that
\bee 
\nabla_\nu H=-R_{\nu\nu}-|\Pi|^2,
\eee	 
where $\nu$ is the unit normal to $ \S$, $R_{\nu\nu}=\Ric(\nu,\nu)$ is the Ricci tensor of $g$ and $\Pi$ is the second fundamental form of $\S$. 
By the Gauss equation, we have
\be
16\pi \rho=\Phi_0(g,K)=R(g_\S)-|\Pi|^2-H^2-2\nabla_\nu H+(\tr_gK)^2-|K|^2,
\ee 
where $R(g_\S)$ is the scalar curvature of $\S$ with the induced metric. Therefore, in order to avoid a distributional `Dirac delta' type of spike in $\rho$, one asks that the mean curvature on each side of $\S$ agree. The remaining geometric boundary conditions come from the momentum constraint. In what follows, and indeed throughout the remainder of this article, it will be useful to work in coordinates adapted to $\S$. Let $\nu$ be the unit normal to $\S\cong\p M$ pointing towards infinity and let $\{\p_A\}$ with $A=1,2$ be a frame on $\S$.

From the momentum constraint, we have
\bee 
8\pi J_\nu=\nabla^A(K_{A\nu})-\nabla_\nu(\tr_\S K);
\eee 
a tangential derivative that is bounded, and a normal derivative of $\tr_\S K$. That is, we must ask that $\tr_\S K$ matches on both sides of $\S$ to avoid a distributional spike in $J_\nu$. The momentum constraint also gives (cf. \eqref{eq-momconstexp} below)
\bee 
8\pi J_A=\nabla_\S^B K_{AB}+K_{\nu B} \Pi^{B}_A+K_{\nu A} H+\nabla_\nu K_{A\nu}-\nabla_A(\tr_\S K)-\nabla^\S_AK_{\nu\nu}.
\eee 
As above, due to the term $\nabla_\nu(K_{\nu A})$, we must ask that $\omega^\perp_A:=K_{\nu A}$ match on either side of $\S$ to avoid a distributional spike in $J_A$.
It may be noticed that the dominant energy condition can only be violated if the distributional spikes in $(\rho,J_i)$ decrease $\rho^2-J_iJ^i$. In fact, recent work of Shibuya \cite[Section VI]{Shibuya} shows that the positive mass theorem indeed still holds for such a manifold that is not smooth along $\S$ provided that the distributional spike `jumps the right way', if such a jump exists (cf. \cite{Miao02}).

This motivates us to insist that an admissible extension of $\S$ is an initial data set $(M,g,K)$ with boundary, such that on $\p M$ the quantities $(g_{\p M},H,\omega^\perp_A,\tr_{\p M} K)$ are prescribed by the corresponding quantities on $\S$ in $(\widehat M,\widehat g,\widehat K)$. These geometric boundary conditions were first proposed by Bartnik; indeed the explanation given above and some related discussion can be found in \cite{Bartnik-TsingHua}.

In particular, an admissible extension in the context of the Bartnik mass depends on the geometric boundary data $(\S,g_\S,H,\omega^\perp_A,\tr_\S K)$. An admissible extension of $\S$ (or of $(\S,g_\S,H,\omega^\perp_A,\tr_\S K)$) is an asymptotically flat initial data set $(M,g,K)$, containing no apparent horizons, whose boundary data agrees with $(\S,g_\S,H,\omega^\perp_A,\tr_\S K)$. The condition that the extension contains no apparent horizons is required in order to exclude extensions where $\S$ is hidden behind a horizon. If this were not excluded then the mass would always be zero, as we could consider extensions where $\S$ is hidden behind an arbitrarily small horizon and the mass could be made arbitrarily small. Taking this to be the definition of an admissible extension, we recall the Bartnik mass is defined to be
\begin{align*} 
\m_B(\S)&=\m_B(\S,g_\S,H,\omega^\perp_A,\tr_\S K)\\
&=\inf\{ \m_{ADM}(M,g,K):(M,g,K) \text{ is an }\textit{admissible extension } \text{of } \S \} .
\end{align*} 

Central to the computations to follow are the linearization of $\Phi$ and its formal adjoint. The linearization of $\Phi$ with respect to $(g,K)$ acts on perturbations $(h,L)$ by
\begin{align}\label{eq-constraints}\begin{split}
	D\Phi_{0\,(g,K)}[h,L]=&\,-\Delta_g(\tr_g h)+\nabla_i\nabla_j h^{ij}-h^{ij}R_{ij}+2h^{ij}K^k_iK_{kj}\\
	&+2\lf( \tr_g(K)\lf( \tr_g(L)-h_{ij}K^{ij} \ri)-L_{ij}K^{ij} \ri)\\
	D\Phi_{i\,(g,K)}[h,L]=&\,2\lf(\nabla^jL_{ij}-h^{jk}\nabla_k  K_{ij} +\nabla_i\lf(\tr_g(L)-h^{jk}K_{jk}\ri)\ri)\\
	&-K^{jk}\nabla_i h_{jk}+K_{ij}\nabla^j\tr_g(h)-2K_{ij}\nabla_k h^{jk}.
	\end{split}
\end{align}

The formal $L^2$-adjoint is then computed by pairing this with some lapse-shift ${\xi=(N,X^i)}$ and formally integrating by parts. This is directly computed as
\begin{align}\label{eq-constraintadjoint}
\begin{split}
	D\Phi^*_{1\,(g,K)}[\xi]=&\,-g^{ij}\Delta_g(N)+\nabla^i\nabla^j N-NR^{ij}+g^{ij}\nabla_k(X^lK_l^k)\\
		&+2NK^{ik} K^j_{k}-2N\tr_g(K)K^{ij}-2X^k\nabla^i K^{j}_k\\
		& +\nabla_k(X^k K^{ij})+\nabla_k (X^k)K^{ij}+2\nabla^i(X^kK_{k}^j)\\
	D\Phi^*_{2\,(g,K)}[\xi]=&\,2N(g^{ij}\tr_g(K)-K^{ij})-2\nabla^jX^i-2g^{ij}\nabla_kX^k,
\end{split}
\end{align}
where the subscripts $1$ and $2$ refer to the components of $D\Phi_{(g,K)}^*[\xi]$ that are paired with $h$ and $L$ respectively.

The Regge-Teitelboim Hamiltonian \cite{R-T} is expressed in terms of the constraint map and a fixed choice of lapse-shift, $\xi$. We fix a choice of $\xi$ that is asymptotic to a constant vector  $\xi_\infty\in\R^{3,1}$. 
We refer readers to \cite[Section 4 and 5]{Bartnik-05} for a precise explanation of the asymptotics required of $\xi$ 
in terms of weighted Sobolev spaces.

The Regge-Teitelboim Hamiltonian is then given by
\be \label{eq-RTHam}
\mathcal{H}(g,K;\xi):=16\pi \mathbb{P}(g,K)\cdot \xi_\infty-\int_M \xi\cdot \Phi(g,K)\sqrt{g},
\ee 
where $\mathbb{P} \in\R^{1,3}$ is the ADM energy momentum co-vector, $\xi_\infty\in\R^{1,3}$ is the asymptotic value of $\xi$, and $\sqrt{g}$ denotes the volume form associated to $g$.

 It is now well-known that \eqref{eq-RTHam} generates the correct equations of motion. Furthermore, results of Moncrief \cite{Moncrief} show that a vacuum spacetime is stationary if and only if, at the initial data level there exists a non-trivial element in the kernel of $D\Phi_{(g,K)}^*$. Note that by a result of Beig and Chru\'sciel \cite{BeigChrusciel} we have that if $D\Phi_{(g,K)}^*[\xi]=0$ then $\xi_\infty$ is parallel to $\mathbb{P}$;
 in particular, if we assume $|\xi_\infty|_{\R^{1,3}}=1$, we have $\xi_{\infty}\cdot\mathbb{P}(g,K)=m_{ADM}(M,g,K)$.

\section{Variation of the Hamiltonian}\label{Svariation}
Our expression for the evolution of the Bartnik mass is derived from an expression of the first variation of the Regge--Teitelboim Hamiltonian on a manifold with boundary. We therefore compute the first variation of the Hamiltonian in this section, and make some geometric interpretations of it.

Let $(M,g,K)$ be vacuum initial data; that is, $\Phi(g,K)=0$. We again fix some lapse-shift $\xi$ on $M$ that is asymptotic to a constant translation $\xi^\mu_\infty=-\frac{1}{m}\mathbb{P}^\mu(g,K)$.
In what follows, we consider \eqref{eq-RTHam} to be defined with respect to this choice of $\xi$. If we formally take the variation of \eqref{eq-RTHam} with respect $g$ and $K$, discarding boundary terms, we then obtain
\bee 
	16\pi \, Dm_{(g,K)}[h,L]-\int_M (h,L)\cdot D\Phi_{(g,K)}^*[\xi]\,\sqrt{g},
\eee 
where we have made use of the fact that $(g,K)$ is vacuum and that 
$$\xi_\infty^\mu\mathbb{P}_\mu(g,K)=m. $$ 
In general though, the boundary terms that we just discarded do not vanish; we must also consider the term

\bee 
	\int_M\lf( (h,L)\cdot D\Phi_{(g,K)}^*[\xi]- \xi\cdot D\Phi_{(g,K)}[h,L]\ri)\,\sqrt{g}.
\eee 

This expression can be divided into two sets of boundary terms; surface integrals at infinity, and surface integrals on the interior boundary $\S$. The boundary terms at infinity cancel exactly with the variation of the mass term, which is indeed motivation for the Regge-Teitelboim Hamiltonian \cite{R-T}. This cancellation is very carefully checked by Bartnik in \cite{Bartnik-05}, and the interested reader is directed there to see the details. In particular, one finds that $D\H_{(g,K;\xi)}[h,L]$ is equal to $-\int_M (h,L)\cdot D\Phi_{(g,K)}^*[\xi]\sqrt{g}$ plus some boundary terms on $\S$. 
We therefore seek a geometric meaning of these boundary terms. The boundary terms can be easily read off from the linearization of the constraint map \eqref{eq-constraints} and its adjoint \eqref{eq-constraintadjoint}, however dealing with all of these terms simultaneously quickly becomes an unwieldy mess. For this reason, we first focus only on the terms containing $N$. These terms are
\be \label{eq-Nterms0}
	\int_{\S}\lf(N(\nabla_i(\tr_g(h))-\nabla_j(h^{j}_i)+h^{j}_i\nabla_j(N)-\tr_g(h)\nabla_iN \ri) \nu^idS,
\ee 
where we take $\nu^i$ to be the unit normal pointing towards infinity. As this computation has been checked in the time-symmetric case and has been considered several times before in the literature, we omit the calculation here for brevity. We simply state that \eqref{eq-Nterms0} can be expressed as
\bee 
	\int_{\S}\nu^i\nabla_i (N)\tr_\S(h)-Nh_{AB}\Pi^{AB}-2NDH_g[h]\,dS,
\eee 
and the interested reader is directed to Proposition 3.7 of \cite{A-K}, for example, to see the computation carried out (see also Lemma 3.1 in \cite{Miao-ICCM}). Note that we again let $\{\p_A\}$ with $A=1,2$ be a frame on $\S$, and $\Pi$ denotes the second fundamental form of $\S$.

This allows us to write the variation of the Hamiltonian as
\begin{align}\begin{split}
	D&\H_{(g,K;\xi)}[h,L]=\, \int_{\S}\nu^i\nabla_i (N)\tr_\S(h)-Nh_{AB}\Pi^{AB}-2NDH_g[h]\,dS\label{eq-DHfull1}\\
	&+\int_{\S} \lf(2X^iL_i^j +X^jK^{ik}h_{ik}-2X^i K^k_ih_{k}^{j}+X^i\tr_g(h)K^j_i -2X^j\tr_g(L) \ri)\nu_j dS.
	\end{split}
\end{align}
It will be useful to split $X$ into components along $\S$ and orthogonal to $\S$, and group terms in \eqref{eq-DHfull1} according to which component of $\xi=(N,X^A,X^\nu)$ they contain. The terms containing $N$ are entirely contained in the first line of \eqref{eq-DHfull1}, so we now proceed to gather the terms containing $X^A$, which are
\begin{align}\begin{split}
	&X^A\lf(2L_A^\nu-2K_A^kh_{k}^\nu + \tr_g(h)K^\nu_A\ri)\label{eq-Aterms1}\\
	&=X^A\lf(2L_A^\nu-\omega^\perp_Ah^\nu_\nu-2K^B_A h^\nu_B+\tr_\S(h)\omega^\perp_A \ri).\end{split}
\end{align}
The terms containing $X^\nu$ are given by
\begin{align} \label{eq-nuterms1}\begin{split}
	&X^\nu\lf( 2L_{\nu\nu}+K^{ik}h_{ik}-2K_\nu^kh_{k\nu}+\tr_g(h)K_{\nu\nu}-2\tr_g(L)  \ri)\\\
	&=X^\nu\lf( -2\tr_\S(L)+K_\S\cdot h_\S+\tr_\S(h)K_{\nu\nu} \ri).\end{split}
\end{align}

Similar to the appearance of $D_gH[h]$ in the terms containing $N$, we hope to write these terms in terms of the variation of the other geometric boundary data, $\omega^\perp$ and $\tr_\S K$. We first compute
\be \label{eq-Domega1}
	D\omega^\perp_{(g,K)}[h,L]=D(K_{iA}\nu^i)_{(g,K)}[h,L]=L_{\nu A}+K_{iA}D(\nu^i)_g[h].
\ee 
The variation of the unit normal vector is computed via the key properties defining it:
\bee 
	g_{ij}\nu^i\nu^j=1 \qquad \text{and} \qquad g_{iA}\nu^i=0.
\eee 

Differentiating these conditions gives
\bee 
	h_{\nu\nu}+2g_{ij}\nu^jD(\nu^i)_g[h]=0
\eee 
and
\bee 
	h_{\nu A}+g_{\nu i} D(\nu^i)_g[h]=h_{\nu A}+g_{AB}D(\nu^B)_g[h]=0.
\eee 

From which we obtain
\bee 
	D(\nu^A)_g[h]=-h_\nu^A \qquad\text{and}\qquad D(\nu^\nu)_g[h]=-\frac12 h_{\nu\nu}.
\eee 

From \eqref{eq-Domega1}, we now have
\bee 
	D\omega^\perp_{(g,K)}[h,L]=L_{\nu A}-K_{AB}h^B_\nu-\frac12K_{A\nu}h_{\nu\nu}.
\eee 

Comparing this to \eqref{eq-Aterms1}, we find that the terms containing $X^A$ are
\bee 
X^A\lf( 2D(\omega^\perp_{A})_{(g,K)}[h,L]+\tr_\S(h)\omega^\perp_A \ri)
\eee 

We now turn to compute
\bee 
	D(\tr_\S K)[h,L]=\tr_\S L+D(g^{AB})_g[h]K_{AB}=\tr_\S L-h^{AB}K_{AB}.
\eee 
 
  We then can write the terms containing $X^\nu$ as
\bee 
	X^\nu\lf( -2D(\tr_\S K)_{(g,K)}[h,L]-K_\S\cdot h_\S +\tr_\S(h)K_{\nu\nu} \ri).
\eee

Bringing this all back together gives us the following expression for the variation of the Hamiltonian:
\begin{align}\nonumber
D&\H_{(g,K;\xi)}[h,L]=\\& 2\int_{\S}-NDH_g[h]+X^AD(\omega^\perp_A)_{(g,K)}[h,L]-X^\nu D(\tr_\S K)_{(g,K)}[h,L]\,dS \nonumber \\\begin{split}
&+\int_{\S}\nu^i\nabla_i (N)\tr_\S(h)-Nh_{AB}\Pi^{AB}+X^A\tr_\S(h)\omega^\perp_A\, dS\label{eq-DHfull}\\
&+\int_{\S}X^\nu\lf(-K_\S\cdot h_\S +\tr_\S(h)K_{\nu\nu} \ri)\end{split}\\
&-\int_M (h,L)\cdot D\Phi^*_{(g,K)}[\xi]\sqrt{g}.\nonumber
\end{align}

We now take a moment to reflect on the various terms in the expression above for the variation of the Hamiltonian. First note that the first line clearly vanishes for all perturbations preserving the geometric boundary data. The second and third lines entirely vanish when $h_{AB}$ is zero; that is, they vanish for all perturbations preserving the metric on the boundary. The fourth line is the only bulk integral, and vanishes if and only if $\xi$ is a Killing vector associated to $(g,K)$.

\section{Evolution of mass formula}\label{SEvoFormula}
We now turn to use the formula derived in the preceding section to derive our formula for the evolution of the Bartnik mass.

Let $(M,g,K)$ be some fixed initial data set and consider a $1$-parameter family of closed surfaces $ \{ \S_t \}$ 
evolving in $M$. Assume there exists an admissible vacuum stationary extension $(M_t,g^S_t,K^S_t)$ of each $\S_t$ that realizes the Bartnik mass of $\S_t$, and that this family of extensions is smooth with respect to $t$.

We assume that the Killing lapse-shift $\xi_t=(N_t,X_t)$ of each stationary extension is asymptotic to a constant translation as in the preceding sections, and scale it so that $N_t^2-|X_t|^2_{g_t}$ is asymptotic to $1$.

Below we recall the statement of Theorem \ref{thm-intro-main} and give its proof.
 
  \begin{thm}\label{thm-main}
	Let $(M,g,K)$ be an initial data set for the Einstein equations. 
	Let $\{ \S_t \}$ be a family of closed, embedded surfaces evolving in  $M$. 
	We assume the evolution is given in terms of a smooth map  $X:\S \times I \to M$ by
	\be
	\frac{d X}{dt}=\eta n .
	\ee 
	Here $I $ is an interval, $n$ is the unit normal pointing towards infinity in $M$, and 
	$\eta $ denotes the  speed of $\Sigma_t = X ( \Sigma, t)$. 
	
	Suppose that for each $\S_t$ there exists an admissible extension (in the sense of Section \ref{SSetup}) $(M_t,g_t,K_t)$ realizing the Bartnik mass of $\S_t$ that is stationary and vacuum. Moreover, suppose $\{ (M_t, g_t, K_t) \}_{t \in I} $  depends smoothly on $t$. Denote by $N_t,X_t^A,X_t^\nu$ the projections of the stationary Killing field orthogonal to the initial data slice, tangential to $\S_t$, and orthogonal to $\S_t$ in $M_t$, respectively.
	
	Then the evolution of the Bartnik mass is given by
	\begin{equation} \label{eq-evoformfull}
	\begin{split}		\frac{d}{dt}\m_B(\S_t)=&\, \frac{1}{16\pi}\int_{\S_t}\eta N_t\lf( |\Pi_t^{(M)}-\Pi_t^{(S)}|^2+|K_{t\,\S}^{(M)}-K_{t\,\S}^{(S)}|^2 \ri)\, d\mu_t \\
	&+\frac{1}{8 \pi}\int_{\S_t} \eta X_t^{\nu}\lf(K_{t\,\S}^{(M)}-K_{t\,\S}^{(S)} \ri)\cdot \lf(\Pi_t^{(M)}-\Pi_t^{(S)} \ri)\,  d\mu_t \\
	&+\int_{\S_t}\eta\lf( N_t\rho+X_t^A J_A+X_t^\nu J_n\ri)\,d\mu_t,
	\end{split}
	\end{equation}
where the superscripts $(S)$ and $(M)$ refer to quantities defined on the stationary extension $M_t$ and on the original manifold $M$ respectively; $\Pi_t$ is the second fundamental form of $\S_t$ in $M_t$; $d\mu_t$ is the volume form of $\S_t$; a subscript ${\S}$ refers to the restriction to $\S$; and $(\rho,J_A,J_n)$ is the energy-momentum covector corresponding to $(M,g,K)$, projected tangentially ($J_A$) and orthogonally ($J_n$) to $\S$ .
\end{thm}

\begin{proof}
Throughout this proof, we use the superscripts ${(S)}$ and ${(M)}$ as described in the statement of the theorem, except for when referencing covariant derivatives. Throughout the computation, covariant derivatives always correspond to the quantities on which they are acting (or it does not matter which of the two connections is used). For example, $\nabla K^{(M)}$ refers to a covariant derivative on $M$.
	
Since we assume each $(M_t,g^S_t,K^S_t)$ is vacuum ($\Phi_\mu(g_t,K_t)=0$), we have
\bee 
\H(g_t,K_t;\xi_t)=16\pi \xi_{t\,\infty}\cdot\mathbb{P}(g_t,K_t)=16\pi\m_B(\S_t),
\eee  
for each $t$. From this we are able to differentiate with respect to $t$ to obtain

\bee 
16\pi \frac{d}{dt}\m_B(\S_t)=D\H_{(g_t,K_t;\xi_t)}[h_t,L_t,f_t],
\eee 
where $h_t=\frac{d}{dt}g_t, L_t=\frac{d}{dt}K_t,f_t=\frac{d}{dt}\xi_t$. We are able to use this to directly compute the variation of the Bartnik mass.

Since the asymptotic value of $\xi$ depends only on $g_t,K_t$ (via $\mathbb{P}$), and since $\Phi(g_t,K_t)=0$, the linearization of $\mathcal{H}$ with respect to $\xi$ vanishes. We therefore simply write
\be \label{eq-dmdh}
16\pi \frac{d}{dt}\m_B(\S_t)=D\H_{(g_t,K_t;\xi_t)}[h_t,L_t].
\ee

The formula for the evolution of mass then follows from the computations in the preceding section, with a bit of extra work. Since each $(M_t,g_t,K_t)$ is stationary, the bulk term in \eqref{eq-DHfull} vanishes, leaving only boundary terms. In what follows, we will omit reference to the parameter $t$ when it is clear from context what we mean. As in the preceding section, we will analyze the remaining terms in \eqref{eq-DHfull} in groups. We begin with the terms containing $N$.

For this, we make use of the fact that we have
\bee 
h=\frac{d}{dt}g=2\eta \Pi^{(M)}
\eee 
and the well-known expression for the evolution of the mean curvature (see \cite[Theorem 3.2]{H-P99}):
\bee 
DH_g[h]=\frac{d}{dt}H=-\Delta_{\S_t}\eta-\eta\lf( |\Pi^{(M)}|^2+\Ric_{ij}^{(M)}n^in^j \ri).
\eee

After integrating by parts, the terms involving $N$ in \eqref{eq-DHfull} can be expressed as
\begin{align}
&\int_{\S_t}2\eta \lf(\Delta_{\S_t}N+H^{(M)}\nu^i\nabla_i N +N\Ric^{(M)}_{ij}n^in^j \ri)\,d\mu\nonumber\\
&+\int_{\S_t}2\eta N\lf(|\Pi^{(M)}|^2-\Pi^{(M)\,AB}\Pi^{(S)}_{AB} \ri)\, d\mu.\label{eq-HevotermsN1}
\end{align}
Note that, by the geometric boundary conditions we have $H^{(M)}=H^{(S)}$ so we simply write $H$. We also omit reference to $M$ and $M_t$ for other quantities that are the same on both by the boundary conditions.

In order to proceed further with the term $\Delta_{\S_t}N+H\nu^i\nabla_i N$, we make use of the identity
\be \label{eq-DeltaSigma}
\Delta_{\S_t}N+H\nu^i\nabla_i N=(g^{ij}-\nu^i\nu^j)\nabla^2_{ij}N
\ee 
combined with the fact that $(M_t,g_t,K_t)$ is stationary. In particular, we make use of the Killing initial data (KID) equations:

\be \label{eq-stateq1}
N\lf(2K_{ik}K^k_j-\Ric_{ij}-K^k_k K_{ij}  \ri)+\L_X K_{ij}+\nabla^2_{ij}N=0
\ee 
and
\be \label{eq-stateq2}
2NK_{ij}+\L_Xg_{ij}=0.
\ee 
Momentarily, we are suppressing the superscript $(S)$, however the following computation is to be understood as entirely on $(M_t,g_t,K_t)$.

We first compute
\begin{align*}
(g^{ij}-\nu^i\nu^j)\L_X K_{ij}=&\,g^{AB}\lf( X^k\nabla_kK_{AB}+2K_{kA}\nabla_BX^k  \ri)\\
=&\,g^{AB}\lf(X^C\nabla_CK_{AB}+X^\nu\nabla_\nu K_{AB}\ri)\\
&+2\lf(K^{AB}\nabla_AX_B+K^{\nu A}\nabla_A X_\nu \ri)\\
=&\,g^{AB}\lf(X^C\nabla_CK_{AB}+X^\nu\nabla_\nu K_{AB}\ri)\\
&-2\lf(N|K_\S|^2-K^{\nu A}\nabla_A X_\nu \ri),
\end{align*} 
where the last equality follows from \eqref{eq-stateq2}. We now turn to compute
\begin{align*}
(g^{ij}-\nu^i\nu^j)&\lf(2K_{ik}K^k_j-\Ric_{ij}-K^k_k K_{ij}  \ri)=\\
=& \, 2|K|^2-R-(\tr_gK)^2-2K_{\nu k}K^k_\nu +\Ric_{\nu\nu}+\tr_g(K)K_{\nu\nu}\\
=&\, 2(|K_\S|^2+2|\omega^\perp|^2+K_{\nu\nu}^2)-R-((\tr_\S K)^2 + 2\tr_\S (K) K_{\nu\nu }  
 +K_{\nu\nu}^2)\\
&-2(|\omega^\perp|^2+K_{\nu\nu}^2)+\Ric_{\nu\nu}+(\tr_\S (K)K_{\nu\nu}+K_{\nu\nu}^2)\\
=&\, 2|K_\S|^2+2|\omega^\perp|^2-R-(\tr_\S K)^2-\tr_\S(K)K_{\nu\nu}+\Ric_{\nu\nu}.
\end{align*}

Now combining these expressions with \eqref{eq-DeltaSigma}, we are able to deal with the term $\Delta_{\S_t}N+H\nu^i\nabla_i N$ appearing in \eqref{eq-HevotermsN1}. In particular, \eqref{eq-HevotermsN1} can be written as
\begin{align}\nonumber
\int_{\S_t}2\eta &N\lf( R^{(S)} -2|\omega^\perp|^2+(\tr_\S K)^2+\tr_\S (K)K^{(S)}_{\nu\nu}+\Ric^{(M)}_{nn}-\Ric^{(S)}_{\nu\nu}\ri)\, d\mu_t\\
\begin{split}
&+\int_{\S_t}2\eta N\lf( |\Pi^{(M)}|^2-\Pi^{(M)\,ij}\Pi^{(S)}_{ij} \ri)\ d\mu_t\label{eq-Nterms2}\\
&-\int_{\S_t}2\eta \lf( 2\omega^{\perp\,A}\nabla_A(X_\nu)+g^{AB}\lf(X^C\nabla_CK^{(S)}_{AB}+X^\nu\nabla_\nu K^{(S)}_{AB} \ri)\ri)\, d\mu_t.\end{split}
\end{align}
For now, we continue to focus on the terms containing $N$ and therefore we will consider only the first two lines in the above expression for now. We will return to the remaining terms containing $X$ later. In order to proceed, the Gauss equation will be required both on $M$ and $M_t$. We have
\begin{align*}
K(\S_t)&=R^{(M)}-2\Ric^{(M)}_{nn}+H^2-|\Pi^{(M)}|^2\\
K(\S_t)&=R^{(S)}-2\Ric^{(S)}_{\nu\nu}+H^2-|\Pi^{(S)}|^2,
\end{align*}
which gives
\bee 
\Ric^{(M)}_{nn}-\Ric^{(S)}_{\nu\nu}=\frac12\lf(R^{(M)}-R^{(S)}+|\Pi^{(S)}|^2-|\Pi^{(M)}|^2 \ri).
\eee 
Substituting this into the first two lines of \eqref{eq-Nterms2} gives
\begin{align}
&\int_{\S_t}2\eta N\lf( R^{(S)} -2|\omega^\perp|^2+(\tr_\S K)^2+\tr_\S (K)K^{(S)}_{\nu\nu}\ri)\, d\mu_t
\nonumber\\
&\label{eq-Nterms3}+\int_{\S_t}2\eta N\lf( |\Pi^{(M)}|^2-\Pi^{(M)\,ij}\Pi^{(S)}_{ij}\ri)d\mu_t\\ &+\int_{\S_t}2\eta N\lf (\frac12\lf(R^{(M)}-R^{(S)}+|\Pi^{(S)}|^2-|\Pi^{(M)}|^2 \ri)\ri)\ d\mu_t\nonumber.
\end{align}
Making use of the Hamiltonian constraint \eqref{eq-constraints} on $(g,K)$, we can write this as
\begin{align*}
&\int_{\S_t}2\eta N\lf(8\pi\rho+\frac12|K_\S^{(M)}|^2 +\frac12(\tr_\S K)^2+\tr_\S(K)(K^{(S)}_{\nu\nu}-K_{nn}^{(M)}  \ri)\,d\mu_t\\
&+\int_{\S_t}\eta N\lf( |\Pi^{(M)}-\Pi^{(S)} |^2+R^{(S)}-2|\omega^\perp|^2\ri).
\end{align*}
Next note that the Hamiltonian constraint on $(g_t,K_t)$, which is vacuum, then gives
\begin{align*}
&\int_{\S_t}2\eta N\lf(8\pi\rho+\frac12|K_\S^{(M)}|^2 +\frac12(\tr_\S K)^2+\tr_\S(K)(K^{(S)}_{\nu\nu}-K_{nn}^{(M)})  \ri)\,d\mu_t\\
&+\int_{\S_t}\eta N\lf( |\Pi^{(M)}-\Pi^{(S)} |^2+(|K^{(S)}_\S|^2-(\tr_\S K)^2-2K_{\nu\nu}^{(S)}\tr_\S(K))\ri)\\
=&\int_{\S_t}\eta N\lf(16\pi\rho+|K_\S^{(M)}-K_\S^{(S)}|^2 +2K_{\S}^{(S)}\cdot K_{\S}^{(M)}-2\tr_\S(K)K_{nn}^{(M)}  \ri)\,d\mu_t\\
&+\int_{\S_t}\eta N\lf(|\Pi^{(M)}-\Pi^{(S)} |^2\ri)\,d\mu_t.
\end{align*}
Recall that this expression is the contribution to $\frac{d}{dt}\m_B(\S_t)$ depending on $N$, so at first glance it may appear to be inconsistent with \eqref{eq-evoformfull}. However, recall that via \eqref{eq-stateq2} we are able to exchange terms containing $N$ and $K$ with terms containing $X$. In particular, we have
\be\label{eq-Statswitch1} 
NK^{(S)}_\S \cdot K^{(M)}_\S = - \nabla^A X^B K^{(M)}_{AB}
\ee 
and
\be \label{eq-Statswitch2}
N\tr_\S (K)=- g^{AB}\nabla_A X_B.
\ee

That is, after making these substitutions, all of the remaining terms containing $N$ agree with those in \eqref{eq-evoformfull}. Unfortunately, we have traded some undesirable terms for a different kind of undesirable term -- we have terms of the form $\nabla X$ to deal with.

We next would like to simplify the terms in \eqref{eq-DHfull} containing $X$. Before doing that, we take a moment to examine the full expression for the derivative of the Bartnik mass after the simplifications made so far:
\begin{align*}
16\pi &\frac{d}{dt}\m_B(\S_t)= \int_{\S_t}\eta N\lf(16\pi\rho+ |\Pi^{(M)}-\Pi^{(S)}|^2+|K_\S^{(M)}-K_\S^{(S)}|^2 \ri)\, d\mu_t \\
&+\int_{\S_t}X^\nu\lf( -2D(\tr_\S K)_{(g,K)}[h,L]-K_\S\cdot h_\S +\tr_\S(h)K_{\nu\nu} \ri)\,d\mu_t\\
&+\int_{\S_t}X^A\lf(2D(\omega^\perp_A)_g[h,L]+\tr_\S(h)\omega^\perp_A\ri)\,d\mu_t\\
&-\int_{\S_t}2\eta\lf( \nabla^AX^BK^{(M)}_{AB}-K^{(M)}_{nn}g^{AB}\nabla_A X_B \ri)\,d\mu_t,\\
&-\int_{\S_t}2\eta \lf( 2\omega^{\perp\,A}\nabla_A(X_\nu)+g^{AB}\lf(X^C\nabla_CK^{(S)}_{AB}+X^\nu\nabla_\nu K^{(S)}_{AB} \ri)\ri)\, d\mu_t,
\end{align*}
where the last line comes from the $X$-terms we dropped from \eqref{eq-Nterms2} and the second last line comes from \eqref{eq-Statswitch1} and \eqref{eq-Statswitch2}. We note that in the second and third lines in the above expression, we are yet to make use of the particular form of of the perturbations $h$ and $L$. 

We focus on the third line, noting that $\tr_\S(h)=2\eta H$, we simply must determine how $\omega^\perp_A=K_{Ai}n^i$ varies along the evolution.

In order to proceed, we consider a point in $M$ where the speed $\eta$ does not vanish. Around this point, the metric $g$ can be expressed in local coordinates as
\bee 
	g=\eta^2dt^2+g_{AB}dx^Adx^B.
\eee 
In determining the evolution of $\omega^\perp$, we are working entirely in the manifold $(M,g)$ so we drop the superscripts $(M)$ for the sake of notational brevity. Now, recall that $D(\omega^\perp_A)_g[h]=\frac{\p}{\p t}\omega^\perp_A$, which we compute as
\begin{align*} 
	\frac{\p}{\p t}\omega^\perp_A&=\frac{\p}{\p t}( \eta^{-1}K_{At} )\\
	&=-\eta^{-2}\eta_{,t}K_{At}+\eta^{-1}K_{At,t}\\
	&=-\nabla_n (\eta)K_{An}+\eta^{-1}\lf(\nabla_t K_{At}+K_{iA}\Gamma^i_{tt}+K_{ti}\Gamma^i_{At}\ri).
\end{align*} 
Now, 
\bee 
	\Gamma^i_{tt}\p_i=\nabla_t(\p_t)=\eta \nabla_n(\eta n)=\eta\nabla_n(\eta)n+\eta^2\nabla_n n,
\eee 
where $\nabla_n n$ can be computed by exploiting the fact that $\nabla_n n$ is tangent to $\S$. We compute
\begin{align*}
	\left<\nabla_n n,\p_A\right>&=-\left<n,\nabla_n\p_A\right>\\
	&=-\left<n,\eta^{-1}\nabla_A\p_t\right>\\
	&=-\left<n,\eta^{-1}\nabla_A(\eta n)\right>\\
	&=-\eta^{-1}\nabla_A \eta.
\end{align*}
We therefore have
\bee 
	\Gamma^{i}_{tt}\p_i=\eta\nabla_n(\eta)n-\eta\nabla_\S \eta;
\eee 
that is,
\bee 
	\Gamma^n_{tt}=\eta\nabla_n(\eta)\qquad \text{and}\qquad \Gamma^A_{tt}=-\eta\nabla^A\eta.
\eee
This gives us
\bee 
	K_{iA}\Gamma^i_{tt}=\eta K_{nA}\nabla_n(\eta )-\eta K_{A}^B\nabla_B \eta.
\eee 
Similarly we have
\bee 
	K_{ti}\Gamma^i_{At}=\eta \nabla_A \eta K_{nn}+\eta^2 K_{n B}\Pi^{B}_A,
\eee 

which allows us to write

\bee 
	\frac{\p}{\p t}\omega^\perp_A=\eta\lf( \nabla_n K_{An}+K_{Bn}\Pi^B_A \ri)-\nabla^B(\eta)K_{AB}+\nabla_A (\eta)K_{nn}.
\eee 

While we computed this for points where $\eta$ does not vanish, it is clear by continuity that this expression is valid everywhere on $\S_t$.

Finally, we turn to compute the evolution of $\tr_\S K$, 
\bee 
	\frac{\p}{\p t} \tr_\S(K)=-2\eta\Pi^{AB}K_{AB}+g^{AB}\p_tK_{AB}.
\eee 

Similar to above, we compute
\begin{align*}
	\p_t K_{AB}&=\nabla_t K_{AB}+K_{iB}\Gamma^i_{At}+K_{iA}\Gamma^i_{Bt}\\
	&=\eta\nabla_n K_{AB}+K_{nB}\nabla_A\eta+K_{nA}\nabla_B\eta+\eta\lf( K_{CB}\Pi^C_A+K_{CA}\Pi^C_B \ri),
\end{align*}
which gives
\bee 
	\frac{\p}{\p t} \tr_\S(K)=\eta\nabla_n(\tr_\S K) +2K_{n}^A\nabla_A\eta.
\eee 

We are now able to interpret each of the terms in \eqref{eq-DHfull} in terms of the evolving surfaces, rather than $h$ and $L$. However, the expression we have for the evolution of quasi-local mass still looks quite far from \eqref{eq-evoformfull}. We collect all of the terms once more, to see what remains:
\begin{align}\nonumber
16\pi \frac{d}{dt}\m_B(\S_t)=\,& \int_{\S_t}\eta N\lf(16\pi\rho+ |\Pi^{(M)}-\Pi^{(S)}|^2+|K_\S^{(M)}-K_\S^{(S)}|^2 \ri)\, d\mu_t  \\
\begin{split} \label{eq-alltermsinterpreted}
&+\int_{\S_t}X^\nu\lf( -2\eta\nabla_n(\tr_\S K^{(M)}) -4K_n^{(M)\,A}\nabla_A\eta\ri)\,d\mu_t\\
& +\int_{\S_t} X^\nu\lf(2\eta HK^{(S)}_{\nu\nu}-2\eta K^{(S)}_\S\cdot \Pi^{(M)}_\S \ri)\,d\mu_t\\
&+\int_{\S_t}2\eta X^A\lf( \nabla_n K^{(M)}_{An}+\omega^\perp_B\Pi^{(M)\,B}_A+H\omega^\perp_A  \ri)\,d\mu_t\\
&+\int_{\S_t}2X^A\lf( \nabla_A (\eta)K^{(M)}_{nn}-\nabla_\S^B(\eta)K^{(M)}_{AB} \ri)\, d\mu_t\\
&-\int_{\S_t}2\eta\lf( \nabla^AX^BK^{(M)}_{AB}-K^{(M)}_{nn}g^{AB}\nabla_A X_B \ri)\,d\mu_t\\
&-\int_{\S_t}2\eta \lf(+g^{AB}\lf(X^C\nabla_CK^{(S)}_{AB}+X^\nu\nabla_\nu K^{(S)}_{AB} \ri)\ri)\, d\mu_t\\
&-\int_{\S_t}4\eta\omega^{\perp\, A}\nabla_A(X_\nu)\,d\mu_t.\end{split}
\end{align}
Note here that we write $\nabla_\S$ (or $\nabla^\S$) to denote the Levi-Civita connection on $(\S,g_\S)$.

Before we continue and examine the $X^\nu$ terms, it will be useful to first group some terms. First we integrate by parts, the terms in the fifth line of \eqref{eq-alltermsinterpreted}. We obtain
\begin{align*}
	\int_{\S_t}2X^A &\lf( \nabla_A(\eta)K_{nn}^{(M)}-K_{AB}^{(M)} \nabla^B \eta\ri)\,d\mu_t\\
	&=\int_{\S_t}2\eta \lf(\nabla_B^\S(K^{(M)\,B}_AX^A) -\nabla^\S_A(X^AK_{nn}^{(M)}) \ri)\,d\mu_t\\
	&=\int_{\S_t}2\eta X^A\lf( \nabla^\S_B(K^{(M)\,B}_A)-\nabla^\S_A(K_{nn}^{(M)} )\ri)\,d\mu_t\\
	&+\int_{\S_t}2\eta \lf(K^{(M)}_{AB}\nabla^B_\S X^A-K_{nn}^{(M)}\nabla^\S_A X^A \ri)\,d\mu_t,
\end{align*}
of which the first integrand will be grouped with the other $X^A$ terms, and the remaining integrand is very closely related to the sixth line of \eqref{eq-alltermsinterpreted}. In particular, we make use of the fact
\bee 
	\nabla^\S_A X_B=\nabla_AX_B-X^\nu \Pi_{AB},
\eee 
to see that the aforementioned terms almost cancel. Making use of this, we can rewrite the expression for the evolution of quasi-local mass as

\begin{align*}
16\pi \frac{d}{dt}\m_B(\S_t)=&\, \int_{\S_t}\eta N\lf(16\pi\rho+ |\Pi^{(M)}-\Pi^{(S)}|^2+|K_\S^{(M)}-K_\S^{(S)}|^2 \ri)\, d\mu_t \\
&+\int_{\S_t}2\eta X^\nu\lf( -\nabla_n(\tr_\S K^{(M)}) - K^{(S)}_\S\cdot \Pi^{(M)}_\S + HK^{(S)}_{\nu\nu} \ri)\,d\mu_t\\
&-\int_{\S_t}2\eta X^\nu\lf(  \Pi^{(S)}\cdot  K^{(M)}_\S   -K_{nn}^{(M)}H \ri)-4X^\nu K_n^{(M)\,A}\nabla_A\eta\,d\mu_t\\
&+\int_{\S_t}2\eta X^A\lf(\lf( \nabla_n K^{(M)}_{An}+\omega^\perp_B\Pi^{(M)\,B}_A \ri)+H\omega^\perp_A\ri)\,d\mu_t\\
&+\int_{\S_t}2\eta X^A\lf( \nabla_\S^B(K^{(M)}_{AB})-\nabla^\S_A (K^{(M)}_{nn})+g^{BC}\nabla_AK_{BC}^{(S)} \ri)\,d\mu_t\\
&-\int_{\S_t}2\eta \lf( 2\omega^{\perp\,A}\nabla_A(X_\nu)+g^{AB}X^\nu\nabla_\nu K^{(S)}_{AB} \ri)\, d\mu_t.
\end{align*}
We would now like to collect all of the terms containing $X^\nu$; that is, the second and third lines in the above expression, as well as the $X^\nu$ terms in the last line. We begin by noting that we should integrate the final term in the third line by parts to obtain
\be 
	4\eta \lf(X^\nu\nabla_A^\S K^{(M)\,A}_n+\omega^{\perp\, A}\nabla^\S_A X^\nu\ri),
\ee 
which can then be written as
\be \begin{split}\label{eq-452}
	4\eta& \left( X^\nu\lf(\nabla_A(K_n^{(M)\,A})-HK_{nn}^{(M)}+K^{(M)}_\S\cdot \Pi^{(M)}\ri)\right.\\&\left.+\omega^{\perp\,A}\nabla_A X^\nu+\omega^{\perp\,A}\Pi^{(S)}_{AB}X^B \right).
	\end{split}
\ee 
Now, the last term in \eqref{eq-452} will be grouped with the $X^A$ terms, and we bring together all of the $X^\nu$ terms now. After factoring out $2\eta X^\nu$, we obtain
\begin{align*}
	-\nabla_n(\tr_\S K^{(M)}) - K^{(S)}_\S\cdot \Pi^{(M)}_\S + HK^{(S)}_{\nu\nu}-\Pi^{(S)}\cdot K^{(M)}_\S +K_{nn}^{(M)}H\\
	+2\nabla_A(K_n^{(M)\,A})-2HK_{nn}^{(M)}+2K^{(M)}_\S\cdot \Pi^{(M)}-g^{AB}\nabla_\nu K_{AB}^{(S)}.
\end{align*}
We can then simplify this using the momentum constraint applied to both $M$ and $M_t$,
\begin{align*} 
	8\pi J_n&=\nabla^A(K_{An}^{(M)})-\nabla_n(\tr_\S K^{(M)})\\
	0&=\nabla^A(K_{A\nu}^{(S)})-\nabla_\nu(\tr_\S K^{(S)}).
\end{align*} 
The $X^\nu$ terms now become
\begin{align*}
8\pi J_n - K^{(S)}_\S\cdot \Pi^{(M)} + H(K^{(S)}_{\nu\nu}-K_{nn}^{(M)})-\Pi^{(S)}\cdot K^{(M)}_\S\\
+\nabla_A(K_n^{(M)\,A})+2K^{(M)}_\S\cdot \Pi^{(M)}-\nabla^A K_{A\nu}^{(S)}.
\end{align*}
We now make use of the fact
\bee 
	\nabla_A K_{Bn}=\nabla^\S_A K_{Bn}+K_{nn}\Pi_{AB}-K_{BC}\Pi^C_A,
\eee 
for both $M$ and $M_t$, to obtain
\begin{align*}
8\pi J_n - K^{(S)}_\S\cdot \Pi^{(M)} -\Pi^{(S)}\cdot K^{(M)}_\S+\nabla^\S_A(K_n^{(M)\,A}-K_\nu^{(S)\,A})\\
+K^{(M)}_\S\cdot \Pi^{(M)}+K^{(S)}_\S\cdot\Pi^{(S)}.
\end{align*}
Finally, making use of $\omega^\perp_A=K^{(M)}_{nA}=K^{(S)}_{\nu A}$, we obtain
\bee 
	8\pi J_n+(K^{(M)}_\S-K^{(S)}_\S)\cdot(\Pi^{(M)}-\Pi^{(S)}).
\eee
That is, collecting all of the terms in our evolution equation expression once more, we have
\begin{align*}
16\pi \frac{d}{dt}\m_B(\S_t)=&\, \int_{\S_t}\eta N\lf(16\pi\rho+ |\Pi^{(M)}-\Pi^{(S)}|^2+|K_\S^{(M)}-K_\S^{(S)}|^2 \ri)\, d\mu_t \\
&+\int_{\S_t}2\eta X^\nu\lf( 8\pi J_n+(K^{(M)}_\S-K^{(S)}_\S)\cdot(\Pi^{(M)}-\Pi^{(S)}) \ri)\,d\mu_t\\
&+\int_{\S_t}2\eta X^A\lf( \nabla_n K^{(M)}_{An}+\omega^\perp_B\Pi^{(M)\,B}_A \ri)\,d\mu_t\\
&+\int_{\S_t}2\eta X^A\lf(\nabla_\S^B(K^{(M)}_{AB})-\nabla^\S_A (K^{(M)}_{nn})\ri)\,d\mu_t\\
&+\int_{\S_t}2\eta X^A\lf(H\omega^\perp_A -g^{BC}\nabla_A( K^{(S)}_{BC})+2\omega^{\perp\,B}\Pi^{(S)}_{AB}\ri)\, d\mu_t.
\end{align*}
Recall, it is our hope to interpret the coefficients of $X^A$ in terms of $J_A$, so we expand the momentum constraint as
\be \label{eq-momconstexp}
	8\pi J_A=\nabla_\S^B K^{(M)}_{AB}+\omega^\perp_B \Pi^{(M)\,B}_A+\omega^\perp_A H+\nabla_n K^{(M)}_{An}-\nabla_A(\tr_\S K^{(M)})-\nabla^\S_AK_{nn}^{M}.
\ee 
Remarkably, all of our $X^A$ terms now can simply be written as
\bee 
	8\pi J_A+\lf( \nabla_A(\tr_\S K)-g^{BC}\nabla_A K_{BC}^{(S)}+2\omega^{\perp\,B}\Pi^{(S)}_{AB} \ri).
\eee 
However, we have
\bee 
	g^{BC}\nabla_A K^{(S)}_{BC}=g^{BC}\lf( \nabla^\S_A K_{BC}^{(S)}+K_{\nu C}^{(S)}\Pi^{(S)}_{AB}+K_{\nu B}^{(S)}\Pi^{(S)}_{AC} \ri)
\eee 
so this simply reduces to $8\pi J_A$. The only remaining thing to note is that the formula we have derived depends on $X^\nu$\\

Therefore, we conclude
\begin{align*}
16\pi \frac{d}{dt}\m_B(\S_t)=&\, \int_{\S_t}\eta N\lf( |\Pi^{(M)}-\Pi^{(S)}|^2+|K_\S^{(M)}-K_\S^{(S)}|^2 \ri)\, d\mu_t \\
&+\int_{\S_t}     2\eta X^\nu  (K^{(M)}-K^{(S)})\cdot(\Pi^{(M)}-\Pi^{(S)})  \,d\mu_t\\
&+\int_{\S_t}  \eta  16\pi  \lf(  N \rho +   X^\nu J_n +  X^A J_A \ri)  \, d\mu_t.
\end{align*}
This completes the proof of \eqref{eq-evoformfull}.
\end{proof}

It is clear from the above derivation that, in the context of Theorem \ref{thm-main}, 
if one only assumes there exists a smooth $1$-parameter family of stationary, vacuum, asymptotically flat manifolds $\{ 
 (M_t,g_t,K_t) \}$ whose boundary data agrees with $\Sigma_t $ in $(M, g, K)$ for each $t$, then
\be \label{eq-dmass-stationary}
\begin{split}
& \,  \frac{d}{dt}\m_{ADM} (M_t, g_t, K_t) \\
= &\,  \frac{1}{16\pi} \int_{\S_t}\eta N\lf( |\Pi^{(M)}-\Pi^{(S)}|^2+|K_\S^{(M)}-K_\S^{(S)}|^2 \ri)\, d\mu_t \\
&+   \frac{1}{8 \pi} \int_{\S_t}     \eta X^\nu  (K^{(M)}-K^{(S)})\cdot(\Pi^{(M)}-\Pi^{(S)})  \,d\mu_t\\
&+\int_{\S_t}  \eta    \lf(  N \rho +   X^\nu J_n +  X^A J_A \ri)  \, d\mu_t.
\end{split}
\ee
We would like to bring readers' attention to the recent work of Z. An \cite{An-18}, in which the author proves that, for data $(g,H,\omega^\perp,\tr_gK)$ near the standard data of a round sphere in a time-symmetric slice of the Minkowski spacetime $ \R^{3,1}$, there exists a (locally unique) stationary vacuum extension $(M^S, g^S, K^S)$ that
depends smoothly on the boundary data.
As a result, formula \eqref{eq-dmass-stationary} is applicable to such stationary extensions produced by Z. An in \cite{An-18}.
 
\medskip 
 
\section*{Acknowledgements}
Robert Bartnik is an inspiration, a friend and a mentor to both of us, and it is truly a pleasure to dedicate this article to him on the occasion of his 60th birthday.

%%%%%%%%%%%%%%%%%%%

\end{document}